\documentclass{amsart}
\usepackage{amssymb}
\usepackage{latexsym}
\usepackage{graphicx}
\usepackage{hyperref}

\newcommand{\Z}{{\mathbb Z}}
\newcommand{\R}{{\mathbb R}}

\newcommand{\N}{{\mathbb N}}

\newtheorem{lemma}{Lemma}[section]
\newtheorem{theorem}[lemma]{Theorem}

\newcommand{\nn}{\nonumber}
\newcommand{\be}{\begin{equation}}
\newcommand{\ee}{\end{equation}}

\newcommand{\ti}{\tilde}

\newcommand{\eps}{\varepsilon}

\numberwithin{equation}{section}

\begin{document}

\title{A family of Schr\"odinger operators whose spectrum is an interval}

\author[H.\ Kr\"uger]{Helge Kr\"uger}
\address{Department of Mathematics, Rice University, Houston, TX~77005, USA}
\email{helge.krueger@rice.edu}

\thanks{H.\ K.\ was supported by NSF grant DMS--0800100.}

\date{\today}


\begin{abstract}
 By approximation, I show that the spectrum of the Schr\"odinger operator with
 potential $V(n) = f(n^\rho \pmod 1)$ for $f$ continuous
 and $\rho > 0$, $\rho \notin \N$ is an interval.
\end{abstract}

\maketitle

\section{Introduction}

In this short note, I wish to describe a family of Schr\"odinger
operators on $l^2(\N)$ whose spectrum is an interval. To set
the stage introduce for a bounded sequence $V: \N\to\R$ and
$u \in l^2(\N)$ the Schr\"odinger operator $H_V$ defined
by
\begin{align}
 (H_V u) (n) & = u(n+1) + u(n-1) + V(n) u(n) & n\geq 2\\
 \nn (H_V u) (1) & = u(2) + V(1) u(1).
\end{align}
We will denote by $\sigma(V)$ the spectrum of the operator
$H_V$.

It is well known that if $V(n)$ is a sequence of independent
identically distributed random variables with distribution $\mu$
satisfying $\mathrm{supp}(\mu) = [a,b]$, we
have that for almost every $V$ the spectrum $\sigma(V)$ is $[a - 2, b + 2]$.
For the Almost--Mathieu Operator with potential
\be
 V_{\lambda,\alpha,\theta}(n) = 2 \lambda \cos(2 \pi (n \alpha + \theta)),
\ee
where $\lambda > 0$, $\alpha \notin \mathbb{Q}$, the
set $\sigma(V_{\lambda,\alpha,\theta})$ contains no interval \cite{aj}.

Bourgain conjectured in \cite{bbook}, that if one considers
the potential
\be
 V_{\lambda, \alpha, x, y} (n) = \lambda \cos \left(\frac{n (n-1)}{2} \alpha + n x + y \right)
\ee
with $\lambda > 0$, $\alpha \notin \mathbb{Q}$, the spectrum 
$\sigma(V_{\lambda, \alpha, x, y})$ is an interval.

Denote by $\mathbb{T} = \mathbb{R} / \mathbb{Z}$ the circle.
I will prove the following result

\begin{theorem}\label{thm:main}
 For any continuous function $f: \mathbb{T} \to \R$,
 any $\alpha\neq 0$, $\theta$, and $\rho > 0$ not an integer,
 introduce the potential
 \be\label{eq:vrho}
  V(n) = f(\alpha n^\rho + \theta).
 \ee
 Then we have that
 \be
  \sigma(V) = [\min(f) - 2, \max(f) + 2].
 \ee
\end{theorem}

Potentials of the type \eqref{eq:vrho}, where already discussed in
Bourgain \cite{b}, Griniasty--Fishman \cite{gf}, Last--Simon \cite{ls}, and Stolz \cite{s}.
In particular, the case $0 < \rho < 1$ is due to Stolz \cite{s}
under an additional regularity assumption on $f$.
The proof of this theorem depends essentially on the following
lemma on the distribution of $n^\rho$,
which is a consequence of a result of Boshernitzan \cite{bosh}.

\begin{lemma}\label{lem:main}
 Let $r \geq 0$ be an integer and $r < \rho < r+1$. Given any
 $\alpha \neq 0$, $\theta$, $K \geq 1$, $\eps > 0$, $a_0, \dots, a_r$, there exists
 an integer $n \geq 1$ such that
 \be
  \sup_{|k| \leq K} \| \alpha (n + k)^\rho + \theta - \sum_{j=0}^r a_j k^j\| \leq \eps,
 \ee
 where $\|x\|= \mathrm{dist}(x, \Z)$ denotes the distance to the closest integer.
\end{lemma}

We will prove this lemma in the next section.
$\| . \|$ is not a norm, but it obeys the triangle inequality
\be\label{eq:triangle}
 \| x + y \| \leq \|x \| + \|y\|
\ee
for any $x,y \in \R$. In particular for any integer $N$, we have that
$\| N x \| \leq |N| \|x\|$.

\begin{proof}[Proof of Theorem~\ref{thm:main}]
 By Lemma~\ref{lem:main}, we can find for any $x \in [0,1)$ a sequence $n_l$
 such that
 $$
  \sup_{|k| \leq l} \| \alpha (n_l + k)^\rho + \theta - x\| \leq 1/l.
 $$
 Hence, the sequence $V_l(n) = V(n - n_l)$ converges pointwise
 to $f(x)$. The claim now follows from a Weyl--sequence argument.
\end{proof}
 
It is remarkable that combined with the Last--Simon semicontinuity
of the absolutely continuous spectrum \cite{ls}, one also obtains
the following result

\begin{theorem}\label{thm:ac}
 For $r \geq 0$ an integer, $r < \rho < r + 1$, and $f$ a continuous
 function on $\mathbb{T}$, introduce the set $\mathcal{B}_r(f)$ as
 \be
  \mathcal{B}_r (f) = \bigcap_{a_0, \dots, a_r} \sigma_{\mathrm{ac}} (f( \sum_{j = 0}^{r} a_j n^j)).
 \ee
 Then for $\alpha\neq 0$ and any $\theta$
 \be\label{eq:ac}
  \sigma_{\mathrm{ac}}( f(\alpha n^\rho + \theta)) \subseteq \mathcal{B}_r(f).
 \ee
 Here $\sigma_{\mathrm{ac}}(V)$ denotes the absolutely continuous
 spectrum of $H_V$.
\end{theorem} 

We note that for $r = 0$, we have that
\be
 \mathcal{B}_0 (f) = [-2 + \max(f), 2 - \min(f)].
\ee
Under additional regularity assumptions on $f$ and $r=0$, Stolz has
shown in \cite{s} that we have equality in \eqref{eq:ac}.

Furthermore note that
\be
 \mathcal{B}_{r+1}(f) \subseteq \mathcal{B}_r(f).
\ee
In \cite{ls}, Last and Simon have stated the following conjecture
\be
 \mathcal{B}_{1} (2 \lambda \cos(2 \pi .)) = \emptyset
\ee
for $\lambda > 0$. They phrased this in poetic terms as
'\textit{Does Hofstadter's Butterfly have wings?}'. The best
positive result in this direction as far as I know, is
by Bourgain \cite{b} showing
\be
 |\mathcal{B}_1 (2 \lambda \cos(2 \pi .))| \to 0, \quad \lambda \to 0.
\ee

It is an interesting question if Theorem~\ref{thm:main} holds
for $\rho \geq 2$ an integer. In the particular case of $f(x) = 2 \lambda \cos(2 \pi x)$, $\rho = 2$,
this would follow from Bourgain's conjecture. However,
there is also the following negative evidence. Consider
the skew-shift $T: \mathbb{T}^2 \to \mathbb{T}^2$ given by
\begin{align}
 T(x,y) &= (x + \alpha, x + y), \\
 \nn T^n(x,y) &= (x + n \alpha, \frac{n(n-1)}{2} \alpha + n x + y),
\end{align}
where $\alpha \notin \mathbb{Q}$. Then Avila, Bochi, and Damanik \cite{abd1}
have shown that for generic continuous $f:\mathbb{T}^2 \to \R$, the spectrum
$\sigma(f(T^n(x,y)))$ contains no interval. So, it is not
clear what to expect in this case.

As a final remark, let me comment on a slight extension.
If one replaces $V$ by the following family of potentials
$$
 V(n) = f(\alpha n^\rho + \sum_{k=1}^{K} \alpha_k n^{\beta_k}),
$$
where $\beta_k < \rho$ and $\alpha_k$ are any numbers,
then Theorem~\ref{thm:main} and \ref{thm:ac} remain valid.

\section{Proof of Lemma~\ref{lem:main}}

Let in the following $r$ be an integer such that
$r < \rho < r +1$. By Taylor expansion, we have that 
\be\label{eq:taylor}
 \alpha (n + k)^\rho = \sum_{j = 0}^{r} x_j(n) k^j + \alpha \frac{\rho \dots (\rho -r)}{(r+1)!} (n + \xi)^{\rho - r - 1} k^{r + 1}
\ee
for some $|\xi| \leq k$ and
\be
 x_j(n) = \alpha \frac{\rho \dots (\rho - j + 1)}{j!} n^{\rho - j}.
\ee
We now first note the following lemma

\begin{lemma}\label{lem:taylor}
 For any $K \geq 1$ and $\eps > 0$, there exists an $N_0(K, \eps)$
 such that
 \be
  |\alpha (n + k)^\rho - \sum_{j = 0}^{r} x_j(n) k^j| \leq \eps
 \ee
 for $|k| \leq K$ and $n \geq N_0(K,\eps)$.
\end{lemma}

\begin{proof}
 This follows from \eqref{eq:taylor} and that $\rho - r - 1 < 0$.
\end{proof}

A sequence $x(n)$ is called uniformly
distributed in $\mathbb{T}^{r+1}$ if for any $0 \leq a_j < b_j \leq 1$, $j = 0, \dots, r$
we have that
\be
 \lim_{n \to \infty} \frac{1}{n} \# \{1 \leq k \leq n:\, x_j(k) \in (a_j, b_j),\quad j=0,\dots,r\} =
 \prod_{j=0}^{r} (b_j - a_j).
\ee
If $x(n)$ is a sequence in $\R^{r+1}$, we can view it as a sequence in $\mathbb{T}^{r+1}$
by considering $x(n) \pmod 1$, and call it uniformly distributed if $x(n) \pmod 1$ is.
We need the following consequence of Theorem~1.8 in \cite{bosh}. 

\begin{theorem}[Boshernitzan]
 Let $(f_1, \dots, f_s)$ be functions $\R \to \R$ of subpolynomial growth,
 that is, there is an integer $N$ such that
 \be
  \lim_{x \to \infty} f_j(x) x^{-N} = 0,\quad 1 \leq j \leq s.
 \ee
 Then the following two conditions are equivalent
 \begin{enumerate}
  \item The sequence $\{f_1(n), \dots, f_s(n)\}_{n\geq 1}$ is uniformly distributed
   in $\mathbb{T}^s$.
  \item For any $(m_1, \dots, m_s) \in \mathbb{Z}^s \backslash \{0\}$, and for every
   polynomial $p(x)$ with rational coefficients, we have that
   \be
    \lim_{x \to \infty} \frac{ \sum_{j=1}^{s} m_j f_j(x) - p(x)}{\log(x)} = \pm \infty.
   \ee
 \end{enumerate}
\end{theorem}

We will use the following consequence of this theorem

\begin{lemma}\label{lem:ud}
 The sequence
 \be
  x(n) = \begin{pmatrix} x_0(n) & \dots & x_r(n) \end{pmatrix}
 \ee
 is uniformly distributed in $\mathbb{T}^{r+1}$.
\end{lemma}

\begin{proof}
 This follows from the fact that for any polynomial $p(n)$
 and integer vector $(m_0, \dots, m_r)$ we have that
 $$
  |\sum_{j=0}^r m_j x_j(n) - p(n)|
 $$
 grows at least like $n^{\rho - r}$, which grows faster than
 $\log(n)$.
\end{proof}

Now we come to

\begin{proof}[Proof of Lemma~\ref{lem:main}]
 By Lemma~\ref{lem:taylor}, there exists an $N_0$ such that
 $$
  \|\alpha (n + k)^\rho - \sum_{j = 0}^{r} x_j(n) k^j\| \leq \frac{\eps}{2}
 $$
 for any $n \geq N_0$ and $|k| \leq K$. By Lemma~\ref{lem:ud},
 we can now find $n \geq N_0$ such that
 $$
  \| x_l(n) - \ti{a}_l \| \leq \frac{\eps}{2 (r+1) K^l},\quad l=0,\dots,r,
 $$
 where $\ti{a}_0 = a_0 - \theta$ and $\ti{a}_l = a_l$, $l\geq 1$. 
 For $|k| \leq K$, we now have that using \eqref{eq:triangle}
 \begin{align}
  \nn \|\alpha (n + k)^\rho & + \theta - \sum_{j = 0}^{r} a_j k^j\| \\ 
  \nn & \leq \|\alpha (n + k)^\rho - \sum_{j = 0}^{r} x_j(n) k^j\| + \|\sum_{j = 0}^{r} x_j(n) k^j - \sum_{j = 0}^{r} a_j k^j + \theta\| \\
  \label{eq:defal} & \leq \frac{\eps}{2} + \sum_{j = 0}^{r} \| (x_j(n) - \ti{a}_j) k^j \| \\
  \label{eq:integer} &\leq \frac{\eps}{2} + \sum_{j = 0}^{r} |k^j| \| x_j(n) - \ti{a}_j \| \\
  \nn & \leq \frac{\eps}{2} + \sum_{j = 0}^{r} K^j \frac{\eps}{2 (r+1) K^j}  = \eps,
 \end{align}
 where we used the definition of $\ti{a}_l$ in \eqref{eq:defal},
 and that $k$ is an integer in \eqref{eq:integer}.
 This finishes the proof.
\end{proof}

\section*{Acknowledgments}

I am indebted to J. Chaika, D. Damanik, and A. Metelkina for useful discussions,
and the referees for many useful suggestions on how to improve the presentation.

\end{document}